\documentclass{new}
\usepackage[T1]{fontenc}
\usepackage[utf8]{inputenc}  
\usepackage[english]{babel}
\usepackage{newunicodechar}
\usepackage{graphicx}
\usepackage{dcolumn}
\usepackage{bm}
\usepackage{amssymb}
\usepackage{bbm}
\usepackage{dsfont}
\usepackage{bbold}
\usepackage{mathbbol}
\usepackage{amsmath}
\usepackage{enumerate}
\usepackage{enumitem}
\usepackage{url}
\usepackage{amsthm}
\usepackage{layout}
\usepackage{epsfig}
\usepackage{graphicx}
\usepackage{epstopdf}
\usepackage{booktabs}
\usepackage{float}
\usepackage{subfigure}
\usepackage{lipsum}
\usepackage{mathrsfs}
\usepackage{booktabs}
\usepackage{blindtext}
\usepackage{appendix}
\usepackage[hyperindex,breaklinks]{hyperref}
\usepackage{enumerate}
\usepackage{enumitem}
\usepackage{makecell}
\usepackage{booktabs}
\usepackage{multirow}
\usepackage{graphicx}
\usepackage{tabularx}
\usepackage{array}
\usepackage{subfigure}\usepackage{color} 
\allowdisplaybreaks[4]

\def\C{\mathbb{C}}
\def\N{\mathbb{N}}
\def\<{\langle}
\def\>{\rangle}

\numberwithin{equation}{section} %%% Equations numbered by section. If you don't want it, please delete it.

\begin{document}

 \PageNum{1}
 \Volume{201x}{Sep.}{x}{x}
 \OnlineTime{August 15, 201x}
 \DOI{0000000000000000}
 \EditorNote{Received x x, 201x, accepted x x, 201x}

\abovedisplayskip 6pt plus 2pt minus 2pt \belowdisplayskip 6pt
plus 2pt minus 2pt
%%%%%%%%%%%%%%%%
\def\vsp{\vspace{1mm}}
\def\th#1{\vspace{1mm}\noindent{\bf #1}\quad}
\def\proof{\vspace{1mm}\noindent{\it Proof}\quad}
\def\no{\nonumber}
\newenvironment{prof}[1][Proof]{\noindent\textit{#1}\quad }
{\hfill $\Box$\vspace{0.7mm}}
\def\q{\quad} \def\qq{\qquad}
\allowdisplaybreaks[4]

%%%%%%%%%%%%%%%%%%%%%%%%%%%%%%%%%%%%%%%%%%%%%%%%%%%%%%%%%%%%%%%%%%%%%%%%%%%%%%%%%%%%%%%%%%%%%%%
%%-------------------       Beginning of  Author's Definitions       -------------------%%
%%                     Note: You may add your own definitions here.

%%-------------------         the end of  Author's Definitions           -------------------%%

\AuthorMark{B. Zheng, S. Yang, J. Hou and K. He }                             %%%  appear on the head of even pages  %%%

\TitleMark{Violation of Bell-type Inequalities on Mutually-commuting von Neumann Algebra Models}  %%% Running Title, appear on the head of odd pages  %%%

\title{Violation of Bell-type Inequalities on Mutually-commuting von Neumann Algebra Models of Entanglement Swapping Networks        %%%   Main Title of your paper  %%%
\footnote{Shuyuan Yang and Bingke Zheng contribute equally to the work. Supported by the National Natural Science
Foundation of China  (Grant No. 12271394, 12571138)}}                  %%%   the Fund which you are supported by  %%%

\author{Bingke \uppercase{Zheng}}             %%%  1st Author's information   %%%
    {College of Mathematics, Taiyuan University of Technology, Taiyuan, 030024, China}

\author{Shuyuan \uppercase{Yang}}             %%%  1st Author's information   %%%
    {School of Mathematics, North University of China, Taiyuan, 030051, China,
\\School of Cyberspace Science and Technology, Beijing Institute of Technology, Beijing, 100081, China
\\
    E-mail\,$:$ yangshuyuan2000@163.com}

\author{Jinchuan \uppercase{Hou}}             %%%  1st Author's information   %%%
    {College of Mathematics, Taiyuan University of Technology, Taiyuan, 030024, China\\
    E-mail\,$:$ jinchuanhou@aliyun.com}

\author{Kan \uppercase{He}}             %%%  1st Author's information   %%%
    {College of Mathematics, Taiyuan University of Technology, Taiyuan, 030024, China\\
    E-mail\,$:$ hekan@tyut.edu.cn}    

\maketitle

\Abstract{Violation of Bell inequalities in bipartite systems represented by mutually-commuting von Neumann algebras has pioneered the study of vacuum entanglement in algebraic quantum field theory. It is unexpected that the maximal violation of Bell inequality can discover algebraic structures. In the paper, we establish the mutually-commuting von Neumann algebra model for entanglement swapping networks and Bell-type inequalities on this model. It generalizes the bipartite case to the ternary case. These algebras are all general von Neumann algebras, which provide a natural perspective to investigate Bell nonlocality in quantum networks in the infinitely-many-degree-of-freedom setting. We determine various bounds for Bell-type inequalities based on the structure of von Neumann algebras, and identify the algebraic structural conditions required for their violation. Finally, we show that the maximal violation of Bell-type inequalities in entanglement swapping networks can be used to determine partially the type classification of the underlying von Neumann algebras.
}      % the abstract

        % the keywords

\MRSubClass{81R15, 47A63, 47L30, 47N50}      % MR(2000) Subject Classification

\section{Introduction}
In non-relativistic quantum mechanics, Bell nonlocality demonstrates that local measurements performed on one subsystem of a quantum state can instantaneously influence the measurement outcomes on another subsystem, regardless of the spatial separation between them \cite{SCA, JSB, JS1}. Such nonlocal correlations can be detected through Bell inequalities, which serve as constraints that all local correlations must obey \cite{ABN, CJL, Guo, CG, HSA}. It has been demonstrated to offer quantum advantages in various device-independent quantum information tasks, including communication complexity \cite{RH}, quantum key distribution \cite{JLA, LSA}, randomness amplification \cite{S, RAA}, and measurement-based quantum computation \cite{RH1, RDH}.

Meanwhile, motivated by the quantum field theory (QFT), which originates from the study of relativistic quantum mechanics, many novel quantum phenomena in systems with infinitely many degrees of freedom have been discovered \cite{HRDK,FK,ASR,MIPRE,CB1,STal,VLAS}. This differs from the non-relativistic quantum-mechanical setup, which is usually linked to type I von Neumann algebras and relies on the algebraic tensor product as its mathematical framework \cite{KR,LNP,SJ}. These are two distinct models, referred to respectively as the tensor product algebra  (TPA) model and the mutually-commuting von Neumann (observable) algebra (MCvNA) model. In the MCvNA model, there is, in general, no tensor product decomposition of the Hilbert space describing subsystems. However, it should be pointed out that relying solely on the TPA model to discuss quantum information problems has drawbacks \cite{RLJ,FCK,WE}. It  fails to provide a universal framework for accurately describing phenomena in systems with infinite degrees of freedom and the quantum field theory, which requires the language of type III von Neumann algebras. Research on quantum information problems on von Neumann algebras has received significant attention and yielded many meaningful results from a mathematical perspective \cite{HLW1,Ikeda1,cuijianlian1,cuijianlian2,Yinzhi1,Gaoli1,Gaoli2,BDE1,KJD1,MM1,CKL1,JL1,wujinsong1,wujinsong2,wujinsong3}.

In the MCvNA model, the algebra of observables of  quantum systems is described by a von Neumann algebra $\mathcal{M}$, with $\mathcal{M}_\mathcal A$ and $\mathcal{M}_\mathcal B$ being two mutually commuting von Neumann subalgebras of $\mathcal{M}$ such that $(\mathcal{M}_\mathcal A \vee \mathcal{M}_\mathcal B)''=\mathcal{M}$.
Here, $\mathcal{M}''$ denotes the double commutant of $\mathcal{M}$ \cite{JMe,ON,MIPRE}. It has been shown that the mutually-commuting von Neumann algebra model provides a more general framework \cite{MIPRE}. In the 1980s early, Summers et al. first introduced the maximal violation of Bell inequality and proved that its value is bounded by 2$\sqrt{2}$ in the MCvNA model of bipartite systems, with equality attainable if and only if each algebra contains a copy of $\mathcal{M}_2({\Bbb C})$ \cite{summer1}. This shows that Bell nonlocality is not merely a quantum peculiarity but a structural feature encoded in the classification of operator algebras, providing rigorous tools to quantify non-classical correlations in relativistic quantum systems \cite{SR}. Translating these bounds into the vacuum representation of algebraic quantum field theory, they show that tangent wedge algebras are always maximally correlated, whereas strictly spacelike-separated wedges decay exponentially with mass-governed distance \cite{summer2,summer3,summer4,summer5,summer6}.  
These works reveal a novel algebraic invariant, termed the Bell correlation invariant, which distinguishes infinitely many isomorphism classes of pairs of mutually commuting von Neumann algebras and links the maximal violation to the occurrence of the hyperfinite type $\rm II_1$ factor \cite{SR}. This is a pioneering work to make Bell nonlocality in QFT serve as a crucial bridge connecting quantum information science with fundamental physics \cite{SR,SBA}. It provides a rigorous framework for reconciling quantum entanglement with relativistic causality, resolves conceptual challenges such as impossible measurements, and reveals how fundamental symmetries like parity violation affect quantum correlations \cite{DHL,FCR,NY,KSA,MY}.

In contrast to entanglement originating from an individual source, quantum networks comprise numerous small-scale entangled states. Owing to the independence among distinct sources, the correlations emerging from quantum networks exhibit non-convex characteristics that transcend the polytopes associated with single-source entanglement \cite{nc1, nc2, nc3, nc4, nc5,Tava}. To date, Bell-type inequalities in the non-relativistic quantum mechanics have been devised to certify nonlocal correlations across diverse network architectures, such as entanglement-swapping networks \cite{nc1, 5}, chain configurations \cite{8, AMID}, star topologies \cite{9, 10}, polygon structures \cite{11, 13,npj2}, tree-shaped networks \cite{14, 15, 16}, arbitrary acyclic networks \cite{nc3, nc4, any3}, and arbitrary $k$-independent networks \cite{12}. Alternative research directions examine the stronger forms of network nonlocality that surpass hybrid implementations involving classical variables and post-quantum resources \cite{fnn, LYK}. Nevertheless, limited progress has been made concerning the discrimination of correlations produced by different networks and the subsequent identification of underlying quantum network topologies \cite{npj1}.
Recently, the notion of bilocality in an  entanglement swapping network based on the MCvNA model has already been introduced by Ligthart et al. \cite{LLD,LLD1}, and Xu has addressed the inclusion problem between TPA model and MCvNA model in this setting \cite{XuXiang}. However, Bell-type inequalities in the MCvNA model have not yet been established. In this paper, we aim to establish bilocal inequalities within the mutually-commuting von Neumann algebra model and investigate how the degree of their violation is related to the structural properties of the algebras.

\section{Ternary mutually-commuting von Neumann algebra models and entanglement swapping networks}\label{sec_XN}

\begin{figure}
\centering
\includegraphics[scale=0.45]{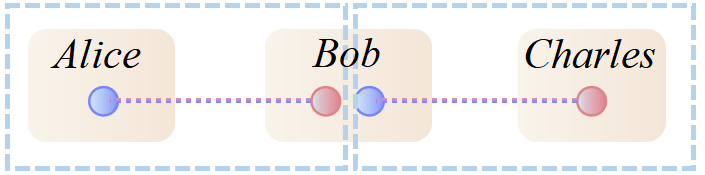}
\caption{An entanglement swapping network scenario with two sources. The connection between two parties represents the sharing of the physical system between them.}
\label{bi}
\end{figure}

{\it Quantum bilocal scenario.}
In non-relativistic quantum mechanics, the quantum entanglement swapping network (see Fig.\,\ref{bi}) is a scenario of three parties consisting of Alice, Bob and Charles, and two sources $\rho_{AB}$, $\rho_{BC}$ shared between them. The inputs and outputs of the measurements performed by the three parties are denoted as $x,y,z$ and $a,b,c$, respectively. Assume that each party performs binary-input and binary-output measurements, with the observables for Alice, Bob, and Charles denoted as $A_x$, $B_y$ and $C_z$, respectively, with  $x, y, z, a, b, c\in \{0,1\}$. Here it is required that the spectra of operators $A_x,B_y,C_z$ are all $\{-1,1\}$, implying that $-I\le A_x,B_y,C_z\le I$.
\iffalse
The measurements are described by positive operator-valued measures (POVMs): for Alice, ${A_{a|x}}$ with outcome $a$; for Bob, ${B_{b|y}}$ with outcome $b$; and for Charles, ${C_{c|z}}$ with outcome $c$. After measurement, they obtain outcomes $a,b,c$, respectively, where $x, y, z\in \{0,1\} \ \text{and}\  a, b, c \in\{-1,1\}$. From these POVMs, one can define the associated observables $A_x = A_{0|x} - A_{1|x}$, $B_y = B_{0|y} - B_{1|y}$, and $C_z = C_{0|z} - C_{1|z}$, satisfying $-I\le A_x,B_y,C_z\le I$. Bob's measurement $y$ may correspond to a joint measurement on the two systems that he receives from each source. 
\fi
The correlations between the measurement outcomes of the three parties are described by the joint probability distribution $p(abc|xyz)$. In this scenario, $p(abc|xyz)$ is said to be bilocal if it can be written as
{\small
\begin{equation*}
p(abc|xyz)=\int\int d\lambda d\mu p_1(\lambda)p_2(\mu)p(a|x,\lambda)p(b|y,\lambda,\mu)p(c|z,\mu),
\end{equation*}}

\noindent
where $\lambda$ and $\mu$ characterize the hidden variables of the systems produced by the sources $\rho_{AB}$ and $\rho_{BC}$, respectively \cite{5,BL2}. Otherwise, it is called non-bilocal. 

In order to detect non-bilocal correlations generated by the network, it is often necessary to find suitable measurements that violate the following bilocal inequality
\begin{equation}\label{bil} \mathcal S\equiv\sqrt{|I|}+\sqrt{|J|}\leq 2, \end{equation}
whose maximum quantum violation is $2\sqrt{2}$ and is attainable. Here
\begin{equation*}
I\equiv\sum_{x,z}\langle A_xB_0C_z\rangle=\langle(A_0+A_1)B_0(C_0+C_1)\rangle,
\end{equation*}
\begin{equation*}
J\equiv\sum_{x,z}(-1)^{x+z}\langle A_xB_1C_z\rangle=\langle(A_0-A_1)B_1(C_0-C_1)\rangle
\end{equation*}
as introduced in Ref.\,\cite{nc5}:
{\small
\begin{eqnarray*}\langle A_xB_yC_z\rangle&=&\sum_{a,b,c=0}^1(-1)^{a+b+c}
{\rm tr}((A_{a|x}B_{b|y}C_{x|z})\rho_{AB}\otimes\rho_{BC})\\
&=&
\sum_{a,b,c=0}^1(-1)^{a+b+c}p(abc|xyz).\end{eqnarray*}} 

\noindent
Here
$A_{x}=\sum_{a}(-1)^aA_{a|x}$, 
$B_{y}=\sum_{b}(-1)^bB_{b|y}$
and
$C_{z}=\sum_{z}(-1)^zC_{c|z}$, where $A_{a|x},\ B_{b|y}$, and $C_{c|z}$ are the positive operator-valued measurements (POVMs) performed by Alice, Bob and Charles, respectively.

{\it Mutually-commuting von Neumann algebra models.}  In QFT, the observables for Alice, Bob, and Charles are associated with three mutually-commuting von Neumann algebras $\mathcal M_\mathcal A$, $\mathcal M_\mathcal B$, $\mathcal M_\mathcal C$. Therefore, our model encompasses both the non-relativistic quantum mechanics scenario and the quantum field theory scenario.
The idea of this model is similar to that in Refs. \cite{LLD,XuXiang}.

\begin{definition}\label{Def1} 
{\bf (Ternary Mutually-commuting von Neumann Algebra Models of Tripartite Quantum Systems)} Let $\mathcal M_\mathcal A, \mathcal M_\mathcal B, \mathcal M_\mathcal C$ be von Neumann subalgebras of $\mathcal B(\mathcal H)$ over some Hilbert space $\mathcal H$, which are mutually commuting, i.e., $\mathcal M_i\subset \mathcal M_j^\prime $ with $i\neq j\in \{\mathcal A, \mathcal B, \mathcal C\}$, where $\mathcal M_j^\prime$ is the communtant of $\mathcal M_j$. The generated von Neumann algebra $$\mathcal M_\mathcal{ABC}=(\mathcal M_\mathcal A \vee \mathcal M_\mathcal B \vee \mathcal M_\mathcal C)^{\prime\prime}.$$
We refer to the above model as the {\bf TMCvNA} model. When $\mathcal M_\mathcal{ABC}\simeq\mathcal M_\mathcal A \otimes \mathcal M_\mathcal B \otimes \mathcal M_\mathcal C$,
it is called the {\bf tensor product algebra model}. In this paper, for any $A \in \mathcal M_\mathcal A$, $B \in \mathcal M_\mathcal B$, $C \in \mathcal M_\mathcal C$, we always assume that they are Hermitian.
\end{definition}

We intend to use the above model of ternary mutually commuting von Neumann algebras to describe the entanglement swapping network in Fig.\,\ref{bi}. We note that there is no correlation between the parties Alice and Charles in the network. Mathematically, this independence can be described by the following formula: a network state $\tau$ of the entanglement swapping network should be a state in $\mathcal{M}_\mathcal{ABC}^*$, the dual space of $\mathcal{M}_\mathcal{ABC}$, satisfying   $$\tau(AC)=\tau(A)\tau(C) \eqno(*)$$ for any $A\in\mathcal M_\mathcal A, C\in\mathcal M_\mathcal C$. We call ($*$)  the independent condition. This assumption will be used throughout this paper.

\begin{definition}\label{Def2}  The ternary mutually-commuting von Neumann algebra model of entanglement swapping networks is the ternary mutually-commuting von Neumann algebra model of tripartite quantum systems with all states satisfying the independence condition ($*$).
\end{definition}

\section{Bilocal inequalities and their bounds}
%the network correlation  $\hat{p}=p(\alpha \beta \gamma|xyz)$ is defined as
%$$p(\alpha \beta \gamma|xyz)=\tau({A}_{\alpha|x}{B}_{\beta|y}{C}_{\gamma|z}),$$
%where  the measurement operators ${A}_{\alpha|x}\in \mathcal M_A,\ {B}_{\beta|y}\in\mathcal M_B,\ {C}_{\gamma|z}\in\mathcal M_C$, $\{x,y,z\}$ and $\{\alpha, \beta, \gamma\}$ denote the inputs and outputs of measurement operators with $x,y,z,\alpha, \beta, \gamma\in\{0,1\}$.
%Let $a_x \in \mathcal M_A$, $b_y \in \mathcal M_B$, $c_z \in \mathcal M_C$ be operators such that $-I \le a_x, b_y, c_z \le I$ for any $x, y, z \in \{0,1\}$, and $a_x=\sum_{\alpha=0}^1(-1)^{\alpha}A_{\alpha|x}, b_y=\sum_{\beta=0}^1(-1)^{\beta}B_{\beta|y}, c_z=\sum_{\gamma=0}^1(-1)^{\gamma}C_{\gamma|z}$.
%Then denote $I_\tau=\tau((a_0+a_1)b_0(c_0+c_1)),\ J_\tau=\tau((a_0-a_1)b_1(c_0-c_1))$. In the TMCvNA model, analogous to Ineq. (\ref{bil}), we define
%\begin{eqnarray}\label{biin}
%  \mathcal S_{\tau}= \sqrt{|I_\tau|}+\sqrt{|J_\tau|}.
%\end{eqnarray}
%Analogous to Ineq. (\ref{g2.3}), under the same assumption, we have the following
%\begin{eqnarray}\label{biin2}
 % \mathcal S_\tau^\prime =  |K_\tau|+|L_\tau|.
%\end{eqnarray}
%where,  $K_\tau=\tau(a_0(b_0+b_1)c_0),\ L_\tau=\tau(a_1(b_0-b_1)c_1)$.\\
In this section, we further analyze the conditions under which the bilocal inequality holds or is violated in the TMCvNA model of an entanglement swapping network. Specifically, in this model, we can construct the bilocal inequality analogous to that in the non-relativistic setting. Let
$$
I_\tau = \tau\bigl((A_0 + A_1) B_0 (C_0 + C_1)\bigr), $$ $$
J_\tau = \tau\bigl((A_0 - A_1) B_1 (C_0 - C_1)\bigr),
$$
where $\tau$ is a state on $\mathcal{M}_\mathcal{ABC}$ satisfying the independent condition ($*$). 
Here, $\tau\bigl(A_xB_yC_z\bigr)=\sum_{a,b,c=0}^1(-1)^{a+b+c}\tau({A}_{a|x}{B}_{b|y}{C}_{c|z})$ and $A_{x}=\sum_{a}(-1)^aA_{a|x}$, 
$B_{y}=\sum_{b}(-1)^bB_{b|y}$,
$C_{z}=\sum_{z}(-1)^zC_{c|z}$,  where $A_{a|x}$, $B_{b|y}$, and $C_{c|z}$ are the POVMs performed by Alice, Bob, and Charles, respectively. Moreover, the network correlation $\hat{p}=p(\alpha \beta \gamma|xyz)$ in the TMCvNA model is defined as
$$p(\alpha \beta \gamma|xyz)=\tau({A}_{\alpha|x}{B}_{\beta|y}{C}_{\gamma|z}),$$
In the TMCvNA model, analogous to Ineq.\,(\ref{bil}), we set
\begin{equation}\label{biin}
\mathcal S_{\tau} = \sqrt{|I_\tau|} + \sqrt{|J_\tau|}.
\end{equation}

We say that the state $\tau$ together with the observables $A_x$, $B_y$, $C_z$ satisfies the bilocal  inequality if $\mathcal{S}_\tau\le2$, and violates it if $\mathcal{S}_\tau>2$.

The following conclusion indicates that in  the TMCvNA of entanglement swapping networks, the supremum of $\mathcal{S}_\tau$ defined in Eq.\,(\ref{biin}) is $2\sqrt{2}$. This coincides with the case in  non-relativistic quantum mechanics, where the bilocal quantity $\mathcal{S}$ in Ineq.\,(\ref{bil}) attains a maximal violation of  $2\sqrt{2}$ allowed by quantum resources.

\begin{theorem}\label{Thm1}
 For the ternary mutually-commuting von Neumann algebra models of entanglement swapping networks, we always have $\mathcal{S}_\tau=\sqrt{|I_\tau|}+\sqrt{|J_\tau|}\le2\sqrt{2}.$
\end{theorem}

\begin{prof}
According to the Gelfand-Namark-Segal (GNS) construction, there is a $^*$-representation $\pi_\tau: \mathcal M_\mathcal{ABC}\rightarrow \mathcal B(\mathcal H_\tau)$ and a cyclic vector $\Omega\in \mathcal H_\tau$ such that the set $\{\pi_\tau(O)\Omega:\  O\in \mathcal M_\mathcal{ABC}\}$ is dense in $\mathcal H_\tau$.
It follows by applying the Cauchy-Schwarz inequality that
\begin{flalign*}
&\mathcal{S}_\tau=\sqrt{|I_\tau|}+\sqrt{|J_\tau|}\\
&=\sqrt{|\tau((A_0+A_1)B_0(C_0+C_1))|}+\sqrt{|\tau((A_0-A_1)B_1(C_0-C_1))|}\\
&\le \sqrt{2}\sqrt{|\tau(B_0(A_0+A_1)(C_0+C_1))|+|\tau(B_1(A_0-A_1)(C_0-C_1))|}\\
&= \sqrt{2}\sqrt{|\<\pi_\tau(B_0)\Omega,\pi_\tau((A_0+A_1)(C_0+C_1))\Omega\>|
+|\<\pi_\tau(B_1)\Omega,\pi_\tau((A_0-A_1)(C_0-C_1))\Omega\>|}\\
&\le \sqrt{2}\sqrt{\|\pi_\tau(B_0)\Omega\|\|\pi_\tau((A_0+A_1)(C_0+C_1))\Omega\|
+\|\pi_\tau(B_1)\Omega\|\|\pi_\tau((A_0-A_1)(C_0-C_1))\Omega\|}\\
&\le \sqrt{2}\sqrt{\sqrt{\tau((A_0+A_1)^2(C_0+C_1)^2)}+\sqrt{\tau((A_0-A_1)^2(C_0-C_1)^2)}}\\
&\le \sqrt{2}\sqrt{\sqrt{\tau(A_0+A_1)^2+\tau(A_0-A_1)^2}\sqrt{\tau(C_0+C_1)^2+\tau(C_0-C_1)^2}}\\
&= \sqrt{2}\sqrt{\sqrt{2\tau(A_0^2+A_1^2)}\sqrt{2\tau(C_0^2+C_1^2)}}\\
&\le \sqrt{2}\sqrt{2\sqrt{2}\sqrt{2}}\\
&=2\sqrt{2}.
\end{flalign*}
The final inequality invokes the condition that $-I\le A_i\le I,\ -I\le C_i\le I$ and the positivity property of $\tau$. 
\end{prof}

%We show the proof in Appendix~\ref{sec-Thm1}.
Building on the results above, we now investigate how the quantity $\mathcal S_\tau$ in Eq.\,(\ref{biin}) depends on the abelianness of the algebras $\mathcal M_\mathcal A$, $\mathcal M_\mathcal B$, and $\mathcal M_\mathcal C$. Specifically, the results indicate that in entanglement swapping networks, the abelianness of the three algebras plays distinct roles in reducing the upper bound of the inequality to 2, i.e., determining the conditions under which no violation of the bilocal inequality can occur.  This is not a simple generalization of the bipartite Bell scenario \cite{SR}, where, with only two systems, Summers et al. showed that if one of these two algebras is abelian, the upper bound of the Bell inequality is 2.

\begin{theorem}\label{Thm2} For the ternary mutually-commuting von Neumann algebra models of entanglement swapping networks, if $\mathcal{M}_\mathcal{A}$ and $\mathcal{M}_\mathcal{C}$ 
are  Abelian, then $$\mathcal{S}_\tau=\sqrt{|I_\tau|}+\sqrt{|J_\tau|}\le 2.$$
\end{theorem}
\begin{prof}
Since $\mathcal M_\mathcal A$ and $\mathcal M_\mathcal C$ are Abelian, respectively, the eight elements
$$A_{\epsilon_0\epsilon_1}\equiv\frac{1}{4}(1+\epsilon_0 A_0)(1+\epsilon_1 A_1),\ \ \ C_{\epsilon_0\epsilon_1}\equiv\frac{1}{4}(1+\epsilon_1 C_0)(1+\epsilon_1 C_1)$$
with $\epsilon_0,\epsilon_1\in\{+,-\}$ are positive. By direct computation, one obtains that$$A_0+A_1=2(A_{++}-A_{--}),\ C_0+C_1=2(C_{++}-C_{--}),$$ $$A_0-A_1=2(A_{+-}-A_{-+}),\ C_0-C_1=2(C_{+-}-C_{-+}).$$ So one obtains that
%\begin{eqnarray*}
%&&\sqrt{|I_\tau|}+\sqrt{|J_\tau|}=\sqrt{|\tau((a_0+a_1)b_0(c_0+c_1))|}+\sqrt{|\tau((a_0-a_1)b_1(c_0-c_1))|}\\
%&=&2\left(\sqrt{|\tau((a_{++}-a_{--})b_0(c_{++}-c_{--}))|}+\sqrt{|\tau((a_{+-}-a_{-+})b_1(c_{+-}-c_{-+}))|}\right)\\
%&=&2\left(\sqrt{|\tau(a_{++}b_0c_{++})-\tau(a_{++}b_0c_{--})-\tau(a_{--}b_0c_{++})+\tau(a_{--}b_0c_{--})|}\right.\\
%&&\left.+\sqrt{|\tau(a_{+-}b_1c_{+-})-\tau(a_{+-}b_1c_{-+})-\tau(a_{-+}b_{1}c_{+-})+\tau(a_{-+}b_1c_{-+})|}\right)\\
%&\le& 2\left(\sqrt{|\tau(a_{++}c_{++})+\tau(a_{++}c_{--})+\tau(a_{--}c_{++})+\tau(a_{--}c_{--})|}\right.\\
%&&\left.+\sqrt{|\tau(a_{+-}c_{+-})+\tau(a_{+-}c_{-+})+\tau(a_{-+}c_{+-})+\tau(a_{-+}c_{-+})|}\right)\\
%&=&2\left(\sqrt{|\tau(a_{++}+a_{--})||\tau(c_{++}+c_{--})|}
%+\sqrt{|\tau(a_{+-}+a_{-+})||\tau(c_{+-}+c_{-+})|}\right)\\
%&\le&2\sqrt{\tau(a_{++}+a_{--})+\tau(a_{+-}+a_{-+})}\sqrt{\tau(c_{++}+c_{--})+\tau(c_{+-}c_{-+})}\\
%&=&2\sqrt{1}\sqrt{1}=2£¬
%\end{eqnarray*}
\begin{flalign*}
&\mathcal{S}_\tau=\sqrt{|I_\tau|}+\sqrt{|J_\tau|}=\sqrt{|\tau((A_0+A_1)B_0(C_0+C_1))|}+\sqrt{|\tau((A_0-A_1)B_1(C_0-C_1))|}&\\
&=2\left(\sqrt{|\tau((A_{++}-A_{--})B_0(C_{++}-C_{--}))|}+\sqrt{|\tau((A_{+-}-A_{-+})B_1(C_{+-}-C_{-+}))|}\right)&\\
&=2\left(\sqrt{|\tau(A_{++}B_0C_{++})-\tau(A_{++}B_0C_{--})-\tau(A_{--}B_0C_{++})+\tau(A_{--}B_0C_{--})|}\right.&\\
&\qquad\left.+\sqrt{|\tau(A_{+-}B_1C_{+-})-\tau(A_{+-}B_1C_{-+})-\tau(A_{-+}B_{1}C_{+-})+\tau(A_{-+}B_1C_{-+})|}\right)&\\
&\le 2\left(\sqrt{|\tau(A_{++}C_{++})+\tau(A_{++}C_{--})+\tau(A_{--}C_{++})+\tau(A_{--}C_{--})|}\right.&\\
&\qquad\left.+\sqrt{|\tau(A_{+-}C_{+-})+\tau(A_{+-}C_{-+})+\tau(A_{-+}C_{+-})+\tau(A_{-+}C_{-+})|}\right)&\\
&=2\left(\sqrt{|\tau(A_{++}+A_{--})||\tau(C_{++}+C_{--})|}
+\sqrt{|\tau(A_{+-}+A_{-+})||\tau(C_{+-}+C_{-+})|}\right)&\\
&\le2\sqrt{\tau(A_{++}+A_{--})+\tau(A_{+-}+A_{-+})}\sqrt{\tau(C_{++}+C_{--})+\tau(C_{+-}C_{-+})}&\\
&=2\sqrt{1}\sqrt{1}=2,
\end{flalign*}
where the first inequality follows the fact that $-I\le B_i\le I\ (i=0,1)$ and the order-preserving of state $\tau$, and the second inequality holds because of the Cauchy-Schwarz inequality and the non-negativeness of $A_{\epsilon_0\epsilon_1} \ \text{and}\ \ C_{\epsilon_0\epsilon_1}$. 
\end{prof}

%Its proof is presented in Appendix~\ref{sec-Thm2}.

The above theorem illustrates a phenomenon: the violation of the Bell-type inequality, i.e., $2<\mathcal{S}_\tau\le2\sqrt{2}$ can serve as an indicator of the non-abelianness of the underlying algebras. Applying the theorem, we can infer from a violation of the inequality in Theorem \ref{Thm2} that at least one of the algebras $\mathcal{M_A}$ and $\mathcal{M_C}$ is non-abelian.

In the following, we define a quantity 
$$\mathcal S(\tau, \mathcal M_\mathcal A, \mathcal M_\mathcal B, \mathcal M_\mathcal C)=\sup_{\{A_x, B_y, C_z\}}(\sqrt{|I_\tau|}+\sqrt{|J_\tau|}).$$

Combining Theorems \ref{Thm1} and \ref{Thm2}, one naturally obtains the following corollary.
%of which a proof is given in Appendix~\ref{sec-Coro4}. 

\begin{corollary}\label{Coro4} In the ternary mutually-commuting von Neumann algebra models of entanglement swapping networks,

{\rm (1)} for any state $\tau$ and any choice of observables $A_x\in\mathcal{M}_\mathcal A,\ B_y\in\mathcal{M}_\mathcal B,\ C_z\in\mathcal{M}_\mathcal C$ in a scheme with two inputs and two outputs, the quantity $\mathcal{S}(\tau,\mathcal{M}_A,\mathcal{M}_B,\mathcal{M}_C)$ satisfies $$2\le \mathcal S(\tau, \mathcal M_\mathcal A, \mathcal M_\mathcal B, \mathcal M_\mathcal C)\le 2\sqrt{2}.$$

{\rm (2)} if $\mathcal M_\mathcal A$ and $\mathcal M_\mathcal C$ are Abelian, then $\mathcal S(\tau, \mathcal M_\mathcal A, \mathcal M_\mathcal B, \mathcal M_\mathcal C)=2$.\\

{\rm (3)} for any states $\phi,\psi\in[(\mathcal{M}_\mathcal A\vee\mathcal{M}_\mathcal B\vee\mathcal{M}_\mathcal C)'']^*$, the following inequality holds: $$|\mathcal S(\phi, \mathcal M_\mathcal A, \mathcal M_\mathcal B, \mathcal M_\mathcal C)-\mathcal S(\psi, \mathcal M_\mathcal A, \mathcal M_\mathcal B, \mathcal M_\mathcal C)|\le k\sqrt{\|\phi-\psi\|},$$ where $k$ is a positive constant. Consequently, the functional $\phi\mapsto \mathcal{S}(\phi,\mathcal{M}_\mathcal A,\mathcal{M}_\mathcal B,\mathcal{M}_\mathcal C)=\displaystyle\sup_{\{A_x, B_y, C_z\}}(\sqrt{|I_\phi|}+\sqrt{|J_\phi|})$ is norm continuous.
\end{corollary}

\begin{prof}
To show (1),  note that  $\sqrt{|I_\tau|}+\sqrt{|J_\tau|}=2$ when $A_0=A_1=B_0=C_0=C_1=I$,  and combining this with the proof of Theorem \ref{Thm1}, we obtain (1).

(2) holds by the fact that $\sqrt{|I_\tau|}+\sqrt{|J_\tau|}=2$ when $A_0=A_1=B_0=C_0=C_1=I$, and by Theorem \ref{Thm2}.

To prove (3),  note that by the representations of $I_\tau$ and $J_\tau$,  together with the facts that $|\sup x-\sup y|\le\sup|x-y|$ and the triangle inequality, one can obtain 
%\begin{eqnarray*}
%&&|\displaystyle\sup_T(\sqrt{|I_\phi|}+\sqrt{|J_\phi|})-\displaystyle\sup_T(\sqrt{|I_\psi|}+\sqrt{|J_\psi|})|\\
%&=&|\displaystyle\sup_T(\sqrt{|\phi((a_0+a_1)b_0(c_0+c_1))|}+\sqrt{|\phi((a_0-a_1)b_1(c_0-c_1))|})\\
%&&-\displaystyle\sup_T(\sqrt{|\psi((a_0+a_1)b_0(c_0+c_1))|}+\sqrt{|\psi((a_0-a_1)b_1(c_0-c_1))|})|\\
%&\le&\displaystyle\sup_T|\sqrt{|\phi((a_0+a_1)b_0(c_0+c_1))|}+\sqrt{|\phi((a_0-a_1)b_1(c_0-c_1))|}\\
%&&-\sqrt{|\psi((a_0+a_1)b_0(c_0+c_1))|}-\sqrt{|\psi((a_0-a_1)b_1(c_0-c_1))|}|\\
%&\le&\displaystyle\sup_T(|\sqrt{|\phi((a_0+a_1)b_0(c_0+c_1))|}-\sqrt{|\psi((a_0+a_1)b_0(c_0+c_1))|}|\\
%&&+|\displaystyle\sqrt{|\phi((a_0+a_1)b_0(c_0+c_1))|}-\sqrt{|\psi((a_0-a_1)b_1(c_0-c_1))|}|)\\
%&\le&\displaystyle\sup_T(\sqrt{|(\phi-\psi)((a_0+a_1)b_0(c_0+c_1))|}+\sqrt{|(\phi-\psi)((a_0-a_1)b_1(c_0-c_1))|})\\
%&\le&\displaystyle\sup_T(\sqrt{\|\phi-\psi\|\|(a_0+a_1)b_0(c_0+c_1)\|}+\sqrt{\|\phi-\psi\|\|(a_0-a_1)b_1(c_0-c_1)\|})\\
%&\le& k\sqrt{\|\phi-\psi\|},
%\end{eqnarray*}
\begin{flalign*}
&|\mathcal S(\phi, \mathcal M_\mathcal A, \mathcal M_\mathcal B, \mathcal M_\mathcal C)-\mathcal S(\psi, \mathcal M_\mathcal A, \mathcal M_\mathcal B, \mathcal M_\mathcal C)|&\\
&=\left|\displaystyle\sup_{\{A_x, B_y, C_z\}}\left(\sqrt{|I_\phi|}+\sqrt{|J_\phi|}\right)-\displaystyle\sup_{\{A_x, B_y, C_z\}}\left(\sqrt{|I_\psi|}+\sqrt{|J_\psi|}\right)\right| &\\
&= \left|\displaystyle\sup_{\{A_x, B_y, C_z\}}\left(\sqrt{|\phi((A_0+A_1)B_0(C_0+C_1))|}+\sqrt{|\phi((A_0-A_1)B_1(C_0-C_1))|}\right)\right. &\\
&\qquad\left.-\displaystyle\sup_{\{A_x, B_y, C_z\}}\left(\sqrt{|\psi((A_0+A_1)B_0(C_0+C_1))|}+\sqrt{|\psi((A_0-A_1)B_1(C_0-C_1))|}\right)\right| &\\
&\le \displaystyle\sup_{\{A_x, B_y, C_z\}}\left|\sqrt{|\phi((A_0+A_1)B_0(C_0+C_1))|}+\sqrt{|\phi((A_0-A_1)B_1(C_0-C_1))|} \right.&\\
&\qquad \left.-\sqrt{|\psi((A_0+A_1)B_0(C_0+C_1))|}-\sqrt{|\psi((A_0-A_1)B_1(C_0-C_1))|}\right| &\\
&\le \displaystyle\sup_{\{A_x, B_y, C_z\}}\left(\left|\sqrt{|\phi((A_0+A_1)B_0(C_0+C_1))|}-\sqrt{|\psi((A_0+A_1)B_0(C_0+C_1))|}\right|\right. &\\
& \qquad\left.+\left|\sqrt{|\phi((A_0+A_1)B_0(C_0+C_1))|}-\sqrt{|\psi((A_0-A_1)B_1(C_0-C_1))|}\right|\right) &\\
&\le \displaystyle\sup_{\{A_x, B_y, C_z\}}\left(\sqrt{|(\phi-\psi)((A_0+A_1)B_0(C_0+C_1))|}+\sqrt{|(\phi-\psi)((A_0-A_1)B_1(C_0-C_1))|}\right) &\\
&\le \displaystyle\sup_{\{A_x, B_y, C_z\}}\left(\sqrt{\|\phi-\psi\|\,\|(A_0+A_1)B_0(C_0+C_1)\|}+\sqrt{\|\phi-\psi\|\,\|(A_0-A_1)B_1(C_0-C_1)\|}\right) &\\
&\le k\sqrt{\|\phi-\psi\|}, &
\end{flalign*}
where the fourth one follows the Cauchy Schwarz inequality, and the last obeys the norm for elements of $\mathcal M_\mathcal A,\mathcal M_\mathcal B,\mathcal M_\mathcal C$ are bounded.
So $\displaystyle\sup_{\{A_x, B_y, C_z\}}(\sqrt{|I_\phi|}+\sqrt{|J_\phi|})$ is norm continuous in the state $\phi$.
\end{prof}

\section{Maximal violation of bilocal inequalities and algebraic structures}

In this section, we aim to show that the violation of bilocal inequalities, in particular the maximal violation, can reflect the structural properties of the algebra. Here, violation refers to exceeding the  bound   for $\mathcal{S}_\tau$ in Eq.\,(\ref{biin}), while maximal violation means attaining the value $2\sqrt{2}$ for the same quantity.

The following theorem analyzes the conditions for maximal violation.
%, with its detailed proof given in Appendix~\ref{sec-Thm5}. 

\begin{theorem}\label{Thm5} For the ternary mutually-commuting von Neumann algebra models of entanglement swapping networks, if the network state $\tau\in\mathcal{M}_\mathcal{ABC}^*$ is faithful, then
$$\sqrt{|I_\tau|}+\sqrt{|J_\tau|}=2\sqrt{2}$$
if and only if $\tau(A_i^2A)=\tau(A),\ \tau(B_i^2B)=\tau(B),\ \tau(C_i^2C)=\tau(C)$, and $\tau[(A_0A_1+A_1A_0)A]=0,\ \tau[(C_0C_1+C_1C_0)C]=0$ for any $A\in\mathcal{M}_\mathcal A,\ B\in\mathcal M_\mathcal B,\  C\in\mathcal{M}_\mathcal C$ with $i\in\{0,1\}$. \\
\end{theorem}

\begin{prof}
If $\sqrt{|I_\tau|}+\sqrt{|J_\tau|}=2\sqrt{2}$, it follows from the proof of Theorem \ref{Thm1} that for any $t\in [0,1]$, these equalities hold:
%$$\left\{\begin{array}{l}
%|\tau(b_0(a_0+a_1)(c_0+c_1))|=|\tau(b_1(a_0-a_1)(c_0-c_1))|\\
%\pi_\tau(b_0)\Omega=k_0\pi_\tau[(a_0+a_1)(c_0+c_1)]\Omega, \ \pi_\tau(b_1)\Omega=k_1\pi_\tau[(a_0-a_1)(c_0-c_1)]\Omega\\
%\|\pi_\tau(b_0)\Omega\|=\|\pi_\tau(b_1)\Omega\|=1\\
%\tau[(a_0+a_1)^2]=t\tau[(c_0+c_1)^2],\ \tau[(a_0-a_1)^2]=t\tau[(c_0-c_1)^2]\\
%a_0^2+a_1^2=2I,c_0^2+c_1^2=2I
%\end{array}\right.$$
\begin{equation}\label{eq1}
|\tau(B_0(A_0+A_1)(C_0+C_1))|=|\tau(B_1(A_0-A_1)(C_0-C_1))|,
\end{equation}
\begin{equation}\label{eq2.1}
\pi_\tau(B_0)\Omega=k_0\pi_\tau[(A_0+A_1)(C_0+C_1)]\Omega,
\end{equation}
\begin{equation}\label{eq2.2}
\pi_\tau(B_1)\Omega=k_1\pi_\tau[(A_0-A_1)(C_0-C_1)]\Omega,\end{equation}
\begin{equation}\label{eq3}
\|\pi_\tau(B_0)\Omega\|=\|\pi_\tau(B_1)\Omega\|=1,
\end{equation}
\begin{equation}\label{eq4.1}
\tau[(A_0+A_1)^2]=t\tau[(C_0+C_1)^2],
\end{equation}
\begin{equation}\label{eq4.2}
\tau[(A_0-A_1)^2]=t\tau[(C_0-C_1)^2],
\end{equation}
\begin{equation}\label{eq5}
A_0^2+A_1^2=2I,C_0^2+C_1^2=2I.
\end{equation}
From  (\ref{eq5}), one gets $A_i^2=I,\ C_i^2=I\ (i=0,1)$ because $-I\le A_i,C_i\le I$, therefore $\tau(A_i^2A)=\tau(A), \ \tau(C_i^2C)=\tau(C)$ for any $A\in\mathcal M_\mathcal A,\ B\in\mathcal M_\mathcal B,\ C\in\mathcal M_\mathcal C$.

Now, let us show the proof for $\tau(B_i^2B)=\tau(B),\ \tau[(A_0A_1+A_1A_0)A]=0$, and $\tau[(C_0C_1+C_1C_0)C]=0$ for any $ A\in\mathcal M_\mathcal A,\ B\in\mathcal M_\mathcal B,\ C\in\mathcal M_\mathcal C$, which implies $ B_i^2=I,$ $A_0A_1+A_1A_0=0,$ and $ C_0C_1+C_1C_0=0$. According to  Eqs.\,(\ref{eq4.1})-(\ref{eq5}),
i.e.,$$\left\{\begin{array}{c}
\tau(A_0^2+A_1^2+A_0A_1+A_1A_0)=t\tau(C_0^2+C_1^2+C_0C_1+C_1C_0)\\
\tau(A_0^2+A_1^2-A_0A_1-A_1A_0)=t\tau(C_0^2+C_1^2-C_0C_1-C_1C_0)\\
A_0^2+A_1^2=2 I, \ C_0^2+C_1^2=2 I,
\end{array}\right.$$
we can get $$t=1,\ \tau(A_0A_1+A_1A_0)=\tau(C_0C_1+C_1C_0).$$
Then combing  conditions (\ref{eq2.1}), (\ref{eq2.2}) and condition (\ref{eq3}),
%\begin{equation}\label{o0}
%\pi_\tau(b_0)\Omega=k_0\pi_\tau[(a_0+a_1)(c_0+c_1)]\Omega,
%\end{equation}
%\begin{equation}
%\pi_\tau(b_1)\Omega=k_1\pi_\tau[(a_0-a_1)(c_0-c_1)]\Omega,\end{equation}
%$$\|\pi_\tau(b_i)\Omega\|=1,$$
one can get
$$
k_0^2\tau[(A_0+A_1)^2(C_0+C_1)^2]=1,
\ k_1^2\tau[(A_0-A_1)^2(C_0-C_1)^2]=1,
$$
 i.e., \begin{equation}\label{k0k1}
     |k_0|=\frac{1}{2+\tau(X)}, \ \ \ |k_1|=\frac{1}{2-\tau(X)}
 \end{equation} because of $\tau(X)=\tau(Y)$, where $X=A_0A_1+A_1A_0,\ Y=C_0C_1+C_1C_0$.
 Then according to Eqs.\,(\ref{eq2.1}), (\ref{eq2.2}), and (\ref{eq1}),
 %\begin{equation}\label{e1}
%|\tau(b_0(a_0+a_1)(c_0+c_1))|=|\tau(b_1(a_0-a_1)(c_0-c_1))|
%\end{equation}
$$|\tau[B_0(A_0+A_1)(C_0+C_1)]|=|k_0||\tau[(A_0+A_1)^2(C_0+C_1)^2]|,$$
$$|\tau[B_1(A_0-A_1)(C_0-C_1)]|=|k_1||\tau[(A_0-A_1)^2(C_0-C_1)^2]|.$$
Substituting (\ref{k0k1}) and (\ref{eq1}) to the above equations, we have $$\frac{1}{2+\tau(X)}(2+\tau(X))^2=\frac{1}{2-\tau(X)}(2-\tau(X))^2$$
deriving $\tau(X)=\tau(Y)=0$. It follows from  $\tau(X)=0=\tau(Y)$ and $\|\pi_\tau(B_i)\Omega\|=1$ that $$k_0^2=k_1^2=\frac{1}{4}$$ from  Eq.\,(\ref{k0k1}).
Since \begin{eqnarray*}
|\langle\Omega,\pi_\tau(B_0^2)\Omega\rangle|&=&|\frac{1}{4}\langle\Omega,\pi_\tau[(A_0+A_1)^2(C_0+C_1)^2]\Omega\rangle|\\
&\le&\frac{1}{4}\sqrt{\|\pi_\tau[(A_0+A_1)(C_0+C_1)]\Omega\|}\sqrt{\|\pi_\tau[(A_0+A_1)(C_0+C_1)]\Omega\|}\\
&=&\frac{1}{4}\tau[(A_0+A_1)^2]\tau[(C_0+C_1)^2]=1,
\end{eqnarray*}
and $\tau(B_0^2)=|\langle\Omega,\pi_\tau(B_0^2)\Omega\rangle|=1$,
this implies that $\pi_\tau[(B_0^2)]\Omega=\Omega$, so $$\tau(B_0^2 B)=\langle\Omega,\pi_\tau(B_0^2)\pi_\tau(B)\Omega\>=\<\pi_\tau(B_0^2)\Omega,\pi_\tau(B)\Omega\>=\<\Omega,\pi_\tau(B)\Omega\>=\tau(B).$$
Similarly $\tau(B_1^2B)=\tau(B)$ for any $B\in\mathcal M_\mathcal B.$
Furthermore, it follows from
Eq.\,(\ref{eq2.1}) and $\pi_\tau[(B_0^2)]\Omega=\Omega$ that $$\pi_\tau(2X+2Y+XY)\Omega=0.$$  So $\tau[(2X+2Y+XY)^2]=0,$ implying that $$\tau(X^2)=\tau(Y^2)=0.$$ Combining the faithfulness, non-negativity of the state $\tau$ and the self-adjointness of $X, \ Y$. Then $X=Y=0$, i.e., $$A_0A_1+A_1A_0=C_0C_1+C_1C_0=0,$$ and implies that $\tau[(A_0A_1+A_1A_0)A]=\tau[(C_0C_1+C_1C_0)C]=0$ for any $A\in\mathcal{M}_\mathcal{A},\ C\in\mathcal{M}_\mathcal{C}$.

It is straightforward to prove the converse process, as we check that Eqs.\,(\ref{eq1})-(\ref{eq5}) hold  if  $\tau(A_i^2A)=\tau(A),\ \tau(B_i^2B)=\tau(B),\ \tau(C_i^2C)=\tau(C)$, and $\tau[(A_0A_1+A_1A_0)A]=0,\ \tau[(C_0C_1+C_1C_0)C]=0$.
 We complete the proof. 
\end{prof}
%{\bf Remark.} This theorem states that when the inequality is maximally violated,  $A_0,\ A_1$ and $-\frac{i}{2}[A_0,A_1]$ form a realization of the Pauli spin matrices in $\mathcal M_\mathcal A$ and generate a copy of $M_2(\C)$ in $\mathcal M_\mathcal A$. The same holds for $C_0,\ C_1$ and $-\frac{i}{2}[C_0,C_1]$ in $\mathcal M_\mathcal C$. Moreover, $\tau(A_0A_1+A_1A_0)=0,\ \tau(C_0C_1+C_1C_0)=0.$\\

To further elucidate the algebraic relations presented in Theorem \ref{Thm5}, we provide the following corollary.
%, whose proof is given in Appendix~\ref{sec-Coro6}.

\begin{corollary}\label{Coro6}
For the ternary mutually-commuting von Neumann algebra models of entanglement swapping networks,  the bilocal inequality can be maximally violated if and only if \( \mathcal{M}_\mathcal{A} \) and \( \mathcal{M}_\mathcal{C} \) contain subalgebras isomorphic to \( M_2(\mathbb C) \) and there exists a  faithful state $\tau\in\mathcal{M}_\mathcal{ABC}^*$ that satisfies the independent condition ($*$): $\tau(AC)=\tau(A)\tau(C)$ for all $A\in\mathcal{M}_\mathcal{A},\ C\in\mathcal{M}_\mathcal{C}$.
\end{corollary}

\begin{prof}
Note that for any von Neumann algebra, there always exists a faithful state $\tau\in\mathcal{M}_{\mathcal{ABC}}^*$.  

$(\Leftarrow)$ Suppose $\mathcal{M}_\mathcal{A}$ (resp. $\mathcal{M}_\mathcal{C}$) contains a subalgebra $\mathcal{M}_\mathcal{A}^{sub}\simeq M_2(\C)$ (resp. $\mathcal{M}_\mathcal{C}^{sub}\simeq M_2(\C)$) and $\tau$ satisfies $(*)$. 

Then there exist operators $A_0,A_1,$ and $A_2:=-\frac{{\rm{i}}}{2}[A_0,A_1]$ in $\mathcal{M}_\mathcal{A}^{sub}$ such that they anticommute and $A_i^2=I$ for $i\in\{0,1,2\}$, where ${\rm{i}}^2=-1$. Consequently, we obtain $$\tau(A_i^2A)=\tau(A),\ \tau[(A_0A_1+A_1A_0)A]=0$$ for any $A\in\mathcal{M}_\mathcal{A}$. 
Similarly for $C\in\mathcal{M}_\mathcal{C}$ with $C_0,C_1,C_2:=-\frac{{\rm{i}}}{2}[C_0,C_1]$. 

Setting 
$B_0=B_1=I$ gives $\tau(B_i^2B)=\tau(B)$ for all $B\in\mathcal{M}_\mathcal{B}$. By Theorem \ref{Thm5}, these operators yield the maximal violation $2\sqrt{2}$ of the quantity $\mathcal{S}_\tau$ in Eq.\,(\ref{biin}). 

$(\Rightarrow)$ Conversely, assume $\mathcal{S}_\tau$ in Eq.\,(\ref{biin}) attains the maximal violation $2\sqrt{2}$. 

Then for any $A\in\mathcal{M}_\mathcal{A},\ B\in\mathcal{M}_\mathcal{B},\ C\in\mathcal{M}_\mathcal{C}$ and $i\in\{0,1\}$, we have $\tau(A_i^2A)=\tau(A),\tau(B_i^2B)=\tau(B),\tau(C_i^2)=\tau(C)$, and $\tau[(A_0A_1+A_1A_0)A]=0, \ \tau[(C_0C_1+C_1C_0)C]=0$.

Taking $A=A_0A_1+A_1A_0$ and using the faithfulness of state $\tau$ together with $\tau[(A_0A_1+A_1A_0)A]=0$, one gets $$A_0A_1+A_1A_0=0,$$ i.e., $A_0A_1=-A_1A_0$. The algebra generated by $A_0,A_1$ is $$\mathfrak{A}(A_0,A_1):=\{\sum_k \alpha_k A_0^mA_1^n|\alpha_k\in\C,m,n\in\N\}.$$ 
From $\tau(A_i^2A)=\tau(A)$ and $-I\le A_i\le I \ (i\in\{0,1\} )$, setting $A=I$ gives $\tau(I-A_i^2)=0$. Faithfulness of state $\tau$ then implies $A_i^2=I$.  The same reasoning yields $B_i^2=I\ \textmd{and} \ C_i^2=I$.

Because $A_i^2=I$, one gets $\sum_k \alpha_k A_0^mA_1^n=\alpha_0 I+\alpha_1 A_0+\alpha_2A_1+\alpha_3A_0A_1$.
Hence,
$$\mathfrak{A}(A_0,A_1)={\rm span}\{I, A_0,A_1,-\frac{{\rm{i}}}{2}[A_0,A_1]\}\simeq M_2(\C).$$ 
Analogously, $\mathfrak{A}(C_0,C_1)\simeq M_2(\C)$. So this proof is completed.
\end{prof}

The above results show a close connection between the maximal violation of Bell-type inequalities and algebraic structures. This connection can be more clearly demonstrated when all algebras are factors.

\begin{corollary}\label{Coro7}
For the ternary tensor product algebra models of entanglement swapping networks, i.e.,  \(  \mathcal{M}_\mathcal{ABC} = \mathcal{M}_\mathcal{A} \otimes   \mathcal{M}_\mathcal{B} \otimes  \mathcal{M}_\mathcal{C}  \), assume that  $\mathcal{M}_\mathcal{A},  \mathcal{M}_\mathcal{B},  \mathcal{M}_\mathcal{C}$ are factors. Then   the bilocal inequality cannot be maximally violated if and only if either \( \mathcal{M}_\mathcal{A} \) or \( \mathcal{M}_\mathcal{C} \) is type {\rm I}$_{2k+1}$ for some finite positive integer $k$.
\end{corollary}

\section{Discussion and Conclusions}\label{sec-conclusion} 

We investigate Bell type inequality for entanglement swapping network in von Neumann algebraic framework, extending the paradigm that Bell violation from observable algebra structure (notably type III factors). We identify algebraic constraints governing inequality violation, linking network nonlocality to the noncommutative structure of the underlying algebras, and further show that maximal violation conditions can reverse-engineer the type classication of von Neumann algebras.

This work represents merely the beginning of a much broader inquiry. Our current model focuses primarily on the simplest nontrivial network: the entanglement swapping scenario with two independent sources. The generalization of these results to arbitrary multipartite quantum networks remains an important open problem. In more complex architectures, the interplay between multiple independent sources and the commutation relations of their associated algebras is expected to reveal even richer structures of nonlocality in networks represented by mutually-commuting von Neumann algebras.

\vspace{1mm}
\th{Conflict of Interest} {\rm The authors declare no conflict of interest.}
\par\vspace{1mm}
%\noindent Please confirm the following statement:

%1) If you are an editorial board member/managing
%editors/editor-in-chief for Acta Mathematica Sinica, English Series, please fill in the following information in 
%this section: ``XX is an editorial board member/managing
%editors/editor-in-chief for Acta Mathematica Sinica, English Series
%and was not involved in the editorial review or the
%decision to publish this article. All authors declare that there are
%no competing interests."

%2) If you are not an editorial board member/managing
%editors/editor-in-chief for Acta Mathematica Sinica, English Series, please fill in the following information in 
%this section: ``The authors declare no conflict of interest. ''

\acknowledgements{\rm We thank the referees for their time and
comments. }

\end{document}